\documentclass[12pt]{article}
\usepackage{txfonts}
\usepackage{indentfirst}
\usepackage[colorlinks,
            linkcolor=blue,
            anchorcolor=green,
            citecolor=blue
            ]{hyperref}

\def\T{\Gamma} \def\D{\Delta} \def\Th{\Theta}
\def\Ld{\Lambda} \def\E{\Sigma} \def\O{\Omega}
\def\a{\alpha} \def\b{\beta} \def\g{\gamma} \def\d{\delta} \def\e{\varepsilon}
\def\r{\rho} \def\o{\sigma} \def\t{\tau} \def\w{\omega} \def\k{\kappa}
\def\th{\theta} \def\ld{\lambda} \def\ph{\varphi} \def\z{\zeta}
\def\A{$A$~} \def\G{$G$~} \def\H{$H$~} \def\K{$K$~} \def\M{$M$~} \def\N{$N$~}
\def\P{$P$~} \def\Q{$Q$~} \def\R{$R$~} \def\V{$V$~} \def\X{$X$~} \def\Y{$Y$~}
\def\rmA{{\bf A}} \def\rmD{{\bf D}} \def\rmS{{\bf S}} \def\rmK{{\bf K}}
\def\rmM{{\bf M}} \def\Z{{\Bbb Z}} \def\GL{{\bf GL}} \def\C{Cayley}
\def\oa{\ovl A} \def\og{\ovl G} \def\oh{\ovl H} \def\ob{\ovl B} \def\oq{\ovl Q}
\def\oc{\ovl C} \def\ok{\ovl K} \def\ol{\ovl L} \def\om{\ovl M} \def\on{\ovl N}
\def\op{\ovl P} \def\oR{\ovl R} \def\os{\ovl S} \def\ot{\ovl T} \def\ou{\ovl U}
\def\ov{\ovl V} \def\ow{\ovl W} \def\ox{\ovl X} \def\oT{\ovl\T}
\def\lg{\langle} \def\rg{\rangle}
\def\di{\bigm|} \def\Di{\Bigm|} \def\nd{\mathrel{\bigm|\kern-.7em/}}
\def\Nd{\mathrel{\not\,\Bigm|}} \def\edi{\bigm|\bigm|}
\def\m{\medskip} \def\l{\noindent} \def\x{$\!\,$}  \def\J{$-\!\,$}
\def\Hom{\hbox{\rm Hom}} \def\Aut{\hbox{\rm Aut}} \def\Inn{\hbox{\rm Inn}}
\def\Syl{\hbox{\rm Syl}} \def\Sym{\hbox{\rm Sym}} \def\Alt{\hbox{\rm Alt}}
\def\Ker{\hbox{\rm Ker}} \def\fix{\hbox{\rm fix}} \def\mod{\hbox{\rm mod}}
\def\psl{{\bf P\!SL}}    \def\sl{\hbox{\rm SL}}   \def\Cay{\hbox{\rm Cay}}
\def\val{\hbox{\rm Val}} \def\Sab{\hbox{\rm Sab}} \def\supp{\hbox{\rm supp}}
\def\qed{\hfill $\Box$}  \def\qqed{\qed\vspace{3truemm}}  \def\Mult{\hbox{\rm Mult}}
\def\CS{\Cay(G,S)} \def\CT{\Cay(G,T)}
\def\h{\heiti\bf} \def\hs{\ziti{E}\bf} \def\st{\songti} \def\ft{\fangsong}
\def\kt{\kaishu} \def\heit{\hs\relax} \def\songt{\st\rm\relax}
\def\fangs{\ft\rm\relax} \def\kaish{\kt\rm\relax} \def\fs{\footnotesize}
\begin{document}

\newtheorem{theorem}{Theorem}[section]
\newtheorem{corollary}[theorem]{Corollary}
\newtheorem{definition}[theorem]{Definition}
\newtheorem{conjecture}[theorem]{Conjecture}
\newtheorem{question}[theorem]{Question}
\newtheorem{lemma}[theorem]{Lemma}
\newtheorem{proposition}[theorem]{Proposition}
\newtheorem{example}[theorem]{Example}
\newtheorem{problem}[theorem]{Problem}
\newenvironment{proof}{\noindent {\bf
Proof.}}{\rule{3mm}{3mm}\par\medskip}
\newcommand{\remark}{\medskip\par\noindent {\bf Remark.~~}}
\newcommand{\pp}{{\it p.}}
\newcommand{\de}{\em}

\def\dfrac{\displaystyle\frac} \def\ovl{\overline}
\def\for{\forall~} \def\exi{\exists~} \def\c{\subseteq}
\def\iif{\Longleftrightarrow} \def\Rto{\Rightarrow} \def\Lto{\Leftarrow}
\def\T{\Gamma} \def\D{\Delta} \def\Th{\Theta}
\def\Ld{\Lambda} \def\E{\Sigma} \def\O{\Omega}
\def\a{\alpha} \def\b{\beta} \def\g{\gamma} \def\d{\delta} \def\e{\varepsilon}
\def\r{\rho} \def\o{\sigma} \def\t{\tau} \def\w{\omega} \def\k{\kappa}
\def\th{\theta} \def\ld{\lambda} \def\ph{\varphi} \def\z{\zeta}
\def\A{$A$~} \def\G{$G$~} \def\H{$H$~} \def\K{$K$~} \def\M{$M$~} \def\N{$N$~}
\def\P{$P$~} \def\Q{$Q$~} \def\R{$R$~} \def\V{$V$~} \def\X{$X$~} \def\Y{$Y$~}
\def\rmA{{\bf A}} \def\rmD{{\bf D}} \def\rmS{{\bf S}} \def\rmK{{\bf K}}
\def\rmM{{\bf M}} \def\Z{{\Bbb Z}} \def\GL{{\bf GL}} \def\C{Cayley}
\def\oa{\ovl A} \def\og{\ovl G} \def\oh{\ovl H} \def\ob{\ovl B} \def\oq{\ovl Q}
\def\oc{\ovl C} \def\ok{\ovl K} \def\ol{\ovl L} \def\om{\ovl M} \def\on{\ovl N}
\def\op{\ovl P} \def\oR{\ovl R} \def\os{\ovl S} \def\ot{\ovl T} \def\ou{\ovl U}
\def\ov{\ovl V} \def\ow{\ovl W} \def\ox{\ovl X} \def\oT{\ovl\T}
\def\lg{\langle} \def\rg{\rangle}
\def\di{\bigm|} \def\Di{\Bigm|} \def\nd{\mathrel{\bigm|\kern-.7em/}}
\def\Nd{\mathrel{\not\,\Bigm|}} \def\edi{\bigm|\bigm|}
\def\m{\medskip} \def\l{\noindent} \def\x{$\!\,$}  \def\J{$-\!\,$}
\def\Hom{\hbox{\rm Hom}} \def\Aut{\hbox{\rm Aut}} \def\Inn{\hbox{\rm Inn}}
\def\Syl{\hbox{\rm Syl}} \def\Sym{\hbox{\rm Sym}} \def\Alt{\hbox{\rm Alt}}
\def\Ker{\hbox{\rm Ker}} \def\fix{\hbox{\rm fix}} \def\mod{\hbox{\rm mod}}
\def\psl{{\bf P\!SL}} \def\Cay{\hbox{\rm Cay}} \def\Mult{\hbox{\rm Mult}}
\def\val{\hbox{\rm Val}} \def\Sab{\hbox{\rm Sab}} \def\supp{\hbox{\rm supp}}
\def\qed{\hfill $\Box$} \def\qqed{\qed\vspace{3truemm}}
\def\CS{\Cay(G,S)} \def\CT{\Cay(G,T)}
\def\h{\heiti\bf} \def\hs{\ziti{E}\bf} \def\st{\songti} \def\ft{\fangsong}
\def\kt{\kaishu} \def\heit{\hs\relax} \def\songt{\st\rm\relax}
\def\fangs{\ft\rm\relax} \def\kaish{\kt\rm\relax} \def\fs{\footnotesize}

\newcommand{\JEC}{{\it Europ. J. Combinatorics},  }
\newcommand{\JCTB}{{\it J. Combin. Theory Ser. B.}, }
\newcommand{\JCT}{{\it J. Combin. Theory}, }
\newcommand{\JGT}{{\it J. Graph Theory}, }
\newcommand{\ComHung}{{\it Combinatorica}, }
\newcommand{\DM}{{\it Discrete Math.}, }
\newcommand{\ARS}{{\it Ars Combin.}, }
\newcommand{\SIAMDM}{{\it SIAM J. Discrete Math.}, }
\newcommand{\SIAMADM}{{\it SIAM J. Algebraic Discrete Methods}, }
\newcommand{\SIAMC}{{\it SIAM J. Comput.}, }
\newcommand{\ConAMS}{{\it Contemp. Math. AMS}, }
\newcommand{\TransAMS}{{\it Trans. Amer. Math. Soc.}, }
\newcommand{\AnDM}{{\it Ann. Discrete Math.}, }
\newcommand{\NBS}{{\it J. Res. Nat. Bur. Standards} {\rm B}, }
\newcommand{\ConNum}{{\it Congr. Numer.}, }
\newcommand{\CJM}{{\it Canad. J. Math.}, }
\newcommand{\JLMS}{{\it J. London Math. Soc.}, }
\newcommand{\PLMS}{{\it Proc. London Math. Soc.}, }
\newcommand{\PAMS}{{\it Proc. Amer. Math. Soc.}, }
\newcommand{\JCMCC}{{\it J. Combin. Math. Combin. Comput.}, }
\newcommand{\GC}{{\it Graphs Combin.}, }

\title{Notes on efficient dominating sets and efficient open dominating sets in circulant graphs\thanks{
\,This research is supported by National Natural Science Foundation of China (Nos. 11401368 and 12326333)
}}
\author{
 Yun-Ping Deng$^{1}$,~~Xuanlong Ma$^{2,}$\footnote{Corresponding author}\\
{\small {\it $^{1}$Department of Mathematics, Shanghai University of Electric Power, Shanghai 201306, China}} \\
{\small  {\it$^{2}$School of Science, Xi'an Petroleum University, Xi'an 710065, China}}\\
 }
\date{}
\maketitle
\newcommand\blfootnote[1]{%
\begingroup
\renewcommand\thefootnote{}\footnote{#1}%
\addtocounter{footnote}{-1}%
\endgroup
}
 \begin{abstract}
Let $D$ be a proper subset of $\mathbb{Z}_n$ such that $|D|$ divides $n$ and is relatively
prime to $\frac{n}{|D|}$.
We first establish a necessary and sufficient condition for $D$ to be an efficient dominating set of some circulant graph $\Cay(\mathbb{Z}_n,S)$.
Moreover, for an inverse-closed subset $S$
of $\mathbb{Z}_n\setminus\{0\}$ such that $|S|+1$ divides $n$ and is relatively prime to $\frac{n}{|S|+1}$, we give a necessary and
sufficient condition for the existence of efficient dominating sets in $\Cay(\mathbb{Z}_n,S)$.
We also
obtain some results about the set of all efficient dominating sets
in any circulant graph. Similar results for efficient open
dominating sets are also obtained.
These results extend some theorems in [Discrete Math. 340 (2017) 1503--1507], [Discrete Math. 340
(2017) 1522--1527], and [Des. Codes Cryptogr. 94 (2026), 87].
\end{abstract}

\blfootnote{E-mail addresses: dyp612@hotmail.com (Deng),
xuanlma@xsyu.edu.cn (Ma)}

\vskip 0.5cm

\section{Introduction}

Let $G$ be a finite group with identity element $e$ and $S$ an
inverse-closed subset of $G$ not containing $e.$ The {\it Cayley
graph} on $G$ with respect to the {\it connection set} $S$ is
defined as the graph with vertex set $G$ and edge set
$\{\{g,sg\}\colon\, g{\in} G,\ s{\in} S\}$ and denoted by
$\Cay(G,S).$ A Cayley graph on finite cyclic group
$\mathbb{Z}_n:=\{0,1,\cdots,n-1\}$ with addition modulo $n$ is
called a {\it circulant graph}.

A subset $D$ of vertices in a graph $\T$ is called an {\it efficient
dominating set} (also called a {\it perfect code}) if $D$ is an
independent set and each vertex not in $D$ is adjacent to exactly
one vertex in $D.$ By the definitions of a circulant graph and
an efficient dominating set, it is easy to see that if
$\Cay(\mathbb{Z}_n,S)$ admits an efficient dominating set $D,$ then
$n=(|S|+1)|D|.$ A subset $D$ of vertices in a graph $\T$ is called
an {\it efficient open dominating set} (also called a {\it total
perfect code}) if each vertex in $\T$ is adjacent to exactly one
vertex in $D.$ Similarly, one can easily see that if
$\Cay(\mathbb{Z}_n,S)$ admits an efficient open dominating set $D,$
then $n=|S||D|,$ and $D$ induces a matching in
$\Cay(\mathbb{Z}_n,S),$ which implies that $|D|$ is even.

The existence of efficient dominating sets in some families of graphs
was first studied by Biggs \cite{Biggs} and Kratochv\'{\i}l
\cite{Kratochvil}. It was proved in \cite{Lee} that a Cayley graph on an
abelian group admits an efficient dominating set if and only if it
is a covering graph of a complete graph. In \cite{Dejter}, a method
for constructing E-chains of Cayley graphs was given and was used to
construct infinite families of E-chains of Cayley graphs on
symmetric groups, where an {\it E-chain} is a countable family of
nested graphs each containing an efficient dominating set.
In \cite{Terada}, Terada showed that there exists no efficient
dominating set in any Cayley graph on special linear group
$\sl(2,2^f)~(f>1)$ with respect to a conjugation-closed connection
set.

During the past two decades, efficient dominating sets in circulant graphs
have attracted considerable attention. For example, Kumar and MacGillivray
\cite{Kumar} characterized the efficient dominating sets in
circulant graphs with domination number two and three, and Deng
\cite{Deng1} further characterized the existence and construction of
efficient dominating sets with domination number prime. In
\cite{Obradovic}, Obradovi\'{c}, Peters and Ru\v{z}i\'{c} gave
necessary and sufficient conditions for the existence of efficient
dominating sets in connected circulant graphs of degree $3$ and $4.$
Applying the tool of cyclotomic polynomials, Feng et al. \cite{Feng}
generalized the results in \cite{Obradovic} by considering the case
of degree $p^l-1$ ($p$ prime and $l$ positive integer). Almost at
the same time, using the subset direct sum method, Deng et al.
\cite{Deng2} independently gave necessary and sufficient conditions
for the existence of efficient dominating sets in connected
non-complete circulant graphs $\Cay(\mathbb{Z}_n,S)$ of degree
$|S|=p-1,pq-1$ and $p^m-1$ ($p,q$ primes and $m$ positive integer)
or order $n=p^kq,p^2q^2,pqr,p^2qr,pqrs$ ($p,q,r,s$ distinct primes
and $k$ positive integer) satisfying that $|S|+1$ divides $n$ and
$|S|+1$ is relatively prime to $\frac{n}{|S|+1}.$
Efficient open dominating sets have also received much attention in
recent years. In \cite[Corollary 3.5 (b)]{HHuang}, Huang et al. gave a
necessary and sufficient condition for a subgroup to be an efficient
open dominating set in circulant graphs. Feng et al. \cite{Feng}
gave a necessary and sufficient condition for the existence of
efficient open dominating sets in connected circulant graphs
$\Cay(\mathbb{Z}_n,S)$ of degree $|S|=p^l$ ($p$ prime and $l$
positive integer) satisfying that $|S|$ divides $n$ and $|S|$ is
relatively prime to $\frac{n}{|S|}$.
In terms of the order of a cyclic group and the cardinality of an efficient dominating set (resp. an efficient open dominating set), Feng et al. \cite{fengh} gave a necessary and sufficient condition  for the existence of efficient dominating sets (resp. efficient open dominating sets) in circulant graphs. Wang et al. \cite{W} studied efficient  dominating sets in circulant graphs of degree $p^{l}-1$, where $p$ is a prime and $l\ge 1$, which extended known results in \cite{Feng} and provided insight on the general
problem on the existence and structure of  efficient dominating sets in circulant graphs. Recently, Cameron et al. \cite{Cam} generalized some results in \cite{Feng} to general abelian groups.

In this note, we
generalize the results in \cite{Cam}, \cite{Deng2}, and \cite{Feng} by proving
that the similar result holds for any circulant graph
$\Cay(\mathbb{Z}_n,S)$ of general degree $|S|$ and general order $n$
satisfying that $|S|+1$ divides $n$ and $|S|+1$ is relatively prime
to $\frac{n}{|S|+1}.$ Furthermore, for any circulant graph which
admits an efficient dominating set, the present paper gives some
results about the set of all its efficient dominating sets.
We further give a necessary
and sufficient condition for a general subset to be an efficient
dominating set in circulant graphs. Similar results for efficient open
dominating sets are also obtained.

\section{Preliminaries}

Let $G$ be a finite abelian group with additive notation, and let $M,N$ be two subsets of \G. If every $g\in G$ has a unique representation in the form $g=m+n$ with $m\in M, n\in N$, then \G is called
the {\it direct sum} of $M$ and $N$.
In this case, we write
$G=M\oplus N$.
For two positive integers $m,n$,
we use $(m,n)$ to denote the greatest common divisor of $m$ and $n$. If $(m,n)=1$, then $m$ and $n$ are relatively prime.

\begin{proposition}
{\rm (\cite[Proposition~2.1]{Vuza})}\label{pr-2.1}
 Let \G be a finite abelian group and let $M,N$ be subsets of \G.
Then $G=M\oplus N$ is equivalent to the conjunction of any two of the following conditions:
\begin{itemize}
  \item[{\rm (i)}] $G=M+N$;
  \item[{\rm (ii)}] $(M-M)\cap(N-N)=\{0\}$;
  \item[{\rm (ii)}] $|G|=|M||N|$.
\end{itemize}
\end{proposition}

If $S$ is a subset of a finite abelian group $G$, then we write $S_0=S\cup\{0\}$.

\begin{lemma} {\rm (\cite[Lemma~2.3]{Deng1})}
\label{yl-2.2}
A subset $D$ of finite abelian group $G$ is an efficient dominating
set of $\Cay(G,S)$ if and only if $G=S_0\oplus D$.
\end{lemma}

For any subset $M$ of an abelian group $G,$ the set of all $g\in G$
satisfying $g+M=M$ is a subgroup of \G, called the {\it stability
subgroup} of $M,$ denoted by $G_M.$ It is easy to see that $M$ is a
union of cosets of $G_M.$ Moreover, if $0\in M,$ then $G_M\subseteq
M.$ A subset $M\subseteq G$ is called {\it periodic} if $G_M\neq
\{0\}.$ Otherwise, a subset $M\subseteq G$ is called {\it aperiodic}
if $G_M= \{0\}.$

\begin{lemma} {\rm (\cite[Lemma~2.4]{Deng1})} \label{yl-2.3}
Let $S$ and $D$ be two subsets of finite abelian group \G, and
let $A=G_{S_0}$. Then $G=S_0\oplus D$ if and only if
$G/A=S_0/A\oplus (D+A)/A$ and $A\cap(D-D)=\{0\}$.
\end{lemma}

\begin{lemma} {\rm (\cite[Theorem~3.17]{Szabo})} \label{yl-2.4}
Let \G be a finite abelian group, and let $M,N$ be subsets of \G. If
$G=M\oplus N,$ then $G=kM\oplus N$ for any integer $k$ relatively
prime to $|M|,$ where $kM=\{km\colon\,m\in M\}.$
\end{lemma}

The following result was first stated in \cite[Remark~1]{Obradovic} and also appears in \cite{Deng1}.

\begin{lemma} {\rm (\cite[Remark~1]{Obradovic})} \label{yl-2.5}
Let $n=|S_0|d,$ where $d$ is an integer. Then $\Cay(\mathbb{Z}_n,S)$
admits an efficient dominating set $D$ which is a coset of some
subgroup of $\mathbb{Z}_n$ if and only if $s-s'\nequiv
0~(mod~|S_0|)$ for any distinct $s,s'\in S_0,$ where
$S_0=S\cup\{0\}.$
\end{lemma}

By the definitions of a Cayley graph and an efficient open dominating
set, similar to Lemma \ref{yl-2.2}, one can easily obtain the
following result.

\begin{lemma} \label{yl-2.6}
A subset $D$ of a finite abelian group $G$ is an efficient open
dominating set of $\Cay(G,S)$ if and only if $G=S\oplus D.$
\end{lemma}

Similar to Lemma \ref{yl-2.5}, we have the following result.

\begin{lemma} \label{yl-2.7}
Let $n=|S|d,$ where $d$ is an integer. Then $\Cay(\mathbb{Z}_n,S)$
admits an efficient open dominating set $D$ which is a coset of some
subgroup of $\mathbb{Z}_n$ if and only if $s-s'\nequiv 0~(mod~|S|)$
for any distinct $s,s'\in S.$
\end{lemma}

\section{Efficient dominating sets}

Clearly,  if
$\Cay(\mathbb{Z}_n,S)$ admits an efficient dominating set $D$,  then $n=(|S|+1)|D|$.
We first establish a necessary and sufficient condition for a proper subset $D$ of $\mathbb{Z}_n$ with $|D|\mid n$ and $(|D|,\frac{n}{|D|})=1$ to be an efficient dominating set of $\Cay(\mathbb{Z}_n,S)$.

\begin{theorem}\label{dl-3.1}
Let $D$ be a proper subset of $\mathbb{Z}_n$ such that $|D|$ divides $n$ and is relatively
prime to $\frac{n}{|D|}$. Then $D$ is an efficient dominating set of some circulant graph
$\Cay(\mathbb{Z}_n,S)$ if and only if $d-d'\nequiv
0~(mod~|D|)$ for distinct $d,d'\in D$.
\end{theorem}

\begin{proof}
Suppose that $d-d'\nequiv 0~(mod~|D|)$ for distinct $d,d'\in D$. Then
$$(|D|\mathbb{Z}_n-|D|\mathbb{Z}_n)\cap(D-D)=\{0\}.$$
Note that the
cardinality of $|D|\mathbb{Z}_n$ is $\frac{n}{|D|}.$ Therefore, by
Proposition \ref{pr-2.1}, $\mathbb{Z}_n=|D|\mathbb{Z}_n\oplus D.$
Now let $S=|D|\mathbb{Z}_n\setminus\{0\}.$ Then $S$ is inverse-closed
and $\mathbb{Z}_n=S_0\oplus D,$ where $S_0=S\cup\{0\}.$ By Lemma~\ref{yl-2.2}, $D$ is an efficient dominating set in
$\Cay(\mathbb{Z}_n,S)$, as desired.

Conversely,
 suppose that $D$ is an efficient dominating set in $\Cay(\mathbb{Z}_n,S)$.
Then by Lemma~\ref{yl-2.2}, $\mathbb{Z}_n=S_0\oplus D,$ where
$S_0=S\cup\{0\}.$ Set $|S_0|=k.$ Then $n=k|D|$ and $k$ is relatively
prime to $|D|.$ It follows from Lemma~\ref{yl-2.4} that
$\mathbb{Z}_n=S_0\oplus kD,$ which implies that $|kD|=|D|.$ Now we
claim that $d-d'\nequiv 0~(mod~|D|)$ for any distinct $d,d'\in D.$
Otherwise, if there exist some distinct $d,d'\in D$ such that
$d-d'\equiv 0~(mod~|D|),$ then $kd-kd'\equiv 0~(mod~k|D|),$ that is,
$kd-kd'\equiv 0~(mod~n).$ Therefore, $kd$ and $kd'$ are the same
element in $\mathbb{Z}_n,$ and thus $|kD|<|D|,$ which is a
contradiction.
\end{proof}

If $S$ is an inverse-closed subset
of $\mathbb{Z}_n\setminus\{0\}$ such that $(|S|+1)\mid n$ and $(|S|+1,\frac{n}{|S|+1})=1$, then the following result gives
a necessary and sufficient condition for which $\Cay(\mathbb{Z}_n,S)$ admits an efficient dominating set.

\begin{theorem}\label{dl-3.2}
Let $S$ be an inverse-closed subset
of $\mathbb{Z}_n\setminus\{0\}$ such that $|S|+1$ divides $n$ and is relatively prime to $\frac{n}{|S|+1}$.  Then $\Cay(\mathbb{Z}_n,S)$ admits an efficient
dominating set if and only if $s-s'\nequiv 0~(mod~|S_0|)$ for
distinct $s,s'\in S_0.$
\end{theorem}
\begin{proof}
If $s-s'\nequiv 0~(mod~|S_0|)$ for distinct $s,s'\in S_0,$ then
by Lemma \ref{yl-2.5}, $\Cay(\mathbb{Z}_n,S)$ admits an efficient
dominating set.

For the converse,
suppose that $\Cay(\mathbb{Z}_n,S)$ admits an efficient dominating set $D$.
Then by Lemma \ref{yl-2.2}, $\mathbb{Z}_n=S_0\oplus D.$ Set $|D|=k.$
Then $n=|S_0|k$ and $k$ is relatively prime to $|S_0|.$ It follows
from Lemma \ref{yl-2.4} that $\mathbb{Z}_n=kS_0\oplus D,$ which
implies $|kS_0|=|S_0|.$ Now we claim that $s-s'\nequiv
0~(mod~|S_0|)$ for any distinct $s,s'\in S_0.$ Otherwise, if there
exist some distinct $s,s'\in S_0$ such that $s-s'\equiv
0~(mod~|S_0|),$ then $ks-ks'\equiv 0~(mod~k|S_0|),$ that is,
$ks-ks'\equiv 0~(mod~n).$ Therefore, $ks$ and $ks'$ are the same
element in $\mathbb{Z}_n,$ and thus $|kS_0|<|S_0|,$ which is a
contradiction.
\end{proof}

In Theorem~\ref{dl-3.2}, if $|S|=p^l-1$ with $p^l\mid n$ and $p^{l+1}\nmid n$ where $p$ is a prime and $l$ is a positive integer, then $(p^l,\frac{n}{p^l})=1$. Thus, Theorem~\ref{dl-3.2} implies \cite[Theorem~1.2]{Feng}.
Moreover, Theorem~\ref{dl-3.2} also generalizes
\cite[Theorem~3.5]{Deng2}.

Now by Lemma~\ref{yl-2.5} and Theorem~\ref{dl-3.2}, one can easily observe the following result.

\begin{corollary} \label{co-3.3}
Let $S$ be an inverse-closed
subset of $\mathbb{Z}_n\setminus\{0\}$ such that $|S|+1$ divides $n$ and is relatively prime to $\frac{n}{|S|+1}.$ Then
$\Cay(\mathbb{Z}_n,S)$ admits an efficient dominating set if and only if it admits an efficient dominating set that is a subgroup of $\mathbb{Z}_n$. In particular, if $\Cay(\mathbb{Z}_n,S)$ admits an efficient dominating set, then every coset of the unique subgroup of order $\frac{n}{|S|+1}$ in $\mathbb Z_n$ is an efficient dominating set.
\end{corollary}

The following theorem gives a complete structural description of efficient dominating sets in $\Cay(\mathbb{Z}_n,S)$, where $|S|+1$ divides $n$ and is relatively prime to $\frac{n}{|S|+1}$.

\begin{theorem}\label{dl-3.4}
Let $S$ be an inverse-closed subset
of $\mathbb{Z}_n\setminus\{0\}$ such that $|S|+1$ divides $n$ and is relatively prime to $\frac{n}{|S|+1}$.
Suppose that
$\Cay(\mathbb{Z}_n,S)$ admits an efficient dominating set.  Then the
set $\mathbb{D}$ of all efficient dominating sets in
$\Cay(\mathbb{Z}_n,S)$ satisfies the following (i)-(ii):

(i) $\{i+|S_0|\mathbb{Z}_n\colon\,i=0,1,\cdots,|S_0|-1\}\subseteq
\mathbb{D}\subseteq\{\{i+k_i\frac{n}{|S_0|}\colon\,i=0,1,\cdots,\frac{n}{|S_0|}-1\}
\colon\,\forall\,k_i\in\{0,1,\cdots,|S_0|-1\}\};$

(ii) For any $D\in\mathbb{D},\,l\in\mathbb{Z}_n,$ and any integer
$k$ relatively prime to $|D|,$ the set $kD+l\in\mathbb{D}.$
\end{theorem}

\begin{proof}
(i) If $\Cay(\mathbb{Z}_n,S)$ admits an efficient dominating set,
then by Theorem \ref{dl-3.2}, $s-s'\nequiv 0~(mod~|S_0|)$ for any
distinct $s,s'\in S_0,$ and thus
$(S_0-S_0)\cap(|S_0|\mathbb{Z}_n-|S_0|\mathbb{Z}_n)=\{0\}.$ Note
that the cardinality of $|S_0|\mathbb{Z}_n$ is $\frac{n}{|S_0|}.$
Therefore, by Proposition \ref{pr-2.1},
$\mathbb{Z}_n=S_0\oplus|S_0|\mathbb{Z}_n.$ By Lemma \ref{yl-2.2},
$|S_0|\mathbb{Z}_n\in\mathbb{D}.$ Since for each
$i\in\{0,1,\cdots,|S_0|-1\},$
$S_0+(i+|S_0|\mathbb{Z}_n)=i+(S_0+|S_0|\mathbb{Z}_n)=i+\mathbb{Z}_n=\mathbb{Z}_n$
and the cardinality of $i+|S_0|\mathbb{Z}_n$ is $\frac{n}{|S_0|},$
it follows from proposition \ref{pr-2.1} that
$\mathbb{Z}_n=S_0\oplus(i+|S_0|\mathbb{Z}_n),$ and thus
$i+|S_0|\mathbb{Z}_n\in\mathbb{D}.$ Therefore,
$\{i+|S_0|\mathbb{Z}_n\colon\,i=0,1,\cdots,|S_0|-1\}\subseteq
\mathbb{D}.$

On the other hand, for any $D\in \mathbb{D},$ by Theorem
\ref{dl-3.1}, $d\nequiv d'~(mod~|D|)$ for any distinct $d,d'\in D.$
Since $|D|=\frac{n}{|S_0|},$ we may write
$D=\{d_0,d_1,\cdots,d_{\frac{n}{|S_0|}-1}\}$ such that $d_i\equiv
i~(mod~\frac{n}{|S_0|})$ for each
$i\in\{0,1,\cdots,\frac{n}{|S_0|}-1\},$ and thus
$d_i=i+k_i\frac{n}{|S_0|},$ where $k_i\in\{0,1,\cdots,|S_0|-1\}.$
Hence
$\mathbb{D}\subseteq\{\{i+k_i\frac{n}{|S_0|}\colon\,i=0,1,\cdots,\frac{n}{|S_0|}-1\}
\colon\,\forall\,k_i\in\{0,1,\cdots,|S_0|-1\}\}.$

(ii) For any $D\in\mathbb{D},$ by Lemma \ref{yl-2.2},
$\mathbb{Z}_n=S_0\oplus D.$ Since $k$ is relatively prime to $|D|,$
it follows from Lemma \ref{yl-2.4} that $\mathbb{Z}_n=S_0\oplus kD,$
and thus $\mathbb{Z}_n=S_0\oplus (kD+l).$ Therefore, by Lemma
\ref{yl-2.2}, $kD+l\in \mathbb{D}.$
\end{proof}

The following theorem extends the previous results to a more general setting and reduces the efficient domination problem in $\mathbb Z_n$ to the corresponding problem in the quotient group.

\begin{theorem}\label{dl-3.5}
Let $S$ be an inverse-closed subset
of $\mathbb{Z}_n\setminus\{0\}$ such that $|S|+1$ divides $n$,  and let $m=(|S_0|,\frac{n}{|S_0|})$.
If $(\frac{|S_0|}{m},\frac{n}{|S_0|})=1$, $S_0$
is periodic with stability subgroup $A$ such that $|A|=m,$ set
$S_0/A=\{A,t_1+A,\cdots,t_{\frac{|S_0|}{m}-1}+A\}$ and
$T_0=\{t_0,t_1,\cdots,t_{\frac{|S_0|}{m}-1}\}$ where $t_0=0$ and
$1\leq t_i\leq \frac{n}{m}-1$ for
$i\in\{1,\cdots,\frac{|S_0|}{m}-1\},$ then $\Cay(\mathbb{Z}_n,S)$
admits an efficient dominating set if and only if $t-t'\nequiv
0~(mod~\frac{|S_0|}{m})$ for distinct $t,t'\in T_0.$
\end{theorem}

\begin{proof}
If $\Cay(\mathbb{Z}_n,S)$ admits an efficient dominating set $D,$
then by Lemmas \ref{yl-2.2} and \ref{yl-2.3},
$\mathbb{Z}_n=S_0\oplus D,$ and thus $\mathbb{Z}_n/A=S_0/A\oplus
(D+A)/A,\,A\cap(D-D)=\{0\},$ it follows that
$\Cay(\mathbb{Z}_n/A,S_0/A\setminus\{A\})$ admits an efficient
dominating set $(D+A)/A,$ that is,
$\Cay(\mathbb{Z}_{\frac{n}{m}},T_0\setminus\{0\})$ admits an
efficient dominating set. Since $|T_0|=\frac{|S_0|}{m}$ and
$\frac{\frac{n}{m}}{|T_0|}=\frac{n}{|S_0|},$ it follows that $|T_0|$
is relatively prime to $\frac{\frac{n}{m}}{|T_0|}.$ By Theorem
\ref{dl-3.2}, $t-t'\nequiv 0~(mod~\frac{|S_0|}{m})$ for any distinct
$t,t'\in T_0.$

If $t-t'\nequiv 0~(mod~\frac{|S_0|}{m})$ for any distinct $t,t'\in
T_0,$ then by Theorem \ref{dl-3.2},
$\Cay(\mathbb{Z}_{\frac{n}{m}},T_0\setminus\{0\})$ admits an
efficient dominating set, that is,
$\Cay(\mathbb{Z}_n/A,S_0/A\setminus\{A\})$ admits an efficient
dominating set, denoted by $(D+A)/A.$ Set
$(D+A)/A=\{d_1+A,\cdots,d_{\frac{n}{|S_0|}}+A\}$ and set
$D=\{d_1,\cdots,d_{\frac{n}{|S_0|}}\}.$ By Lemma \ref{yl-2.2},
$\mathbb{Z}_n/A=S_0/A\oplus (D+A)/A.$ Since
$|(D+A)/A|=|D|=\frac{n}{|S_0|},$ it follows that $A\cap(D-D)=\{0\}.$
By Lemma \ref{yl-2.3}, $\mathbb{Z}_n=S_0\oplus D.$ Therefore, by
Lemma \ref{yl-2.2}, $\Cay(\mathbb{Z}_n,S)$ admits an efficient
dominating set $D.$
\end{proof}

\section{Efficient open dominating sets}

Now we turn to study efficient open dominating sets in circulant graphs.

\begin{theorem}\label{dl-3.6}
Let $D$ be a proper subset of $\mathbb{Z}_n$ such that $|D|$ divides
$n$ and is relatively prime to $\frac{n}{|D|}.$ Then $D$ is an
efficient open dominating set in some circulant graph
$\Cay(\mathbb{Z}_n,S)$ if and only if $|D|$ is even and
$d-d'\nequiv 0~(mod~|D|)$ for any distinct $d,d'\in D.$
\end{theorem}

\begin{proof}
If $|D|$ is even and $d-d'\nequiv 0~(mod~|D|)$ for any distinct
$d,d'\in D,$ then
$((\frac{n}{2}+|D|\mathbb{Z}_n)-(\frac{n}{2}+|D|\mathbb{Z}_n))\cap(D-D)=\{0\}.$
Note that the cardinality of $\frac{n}{2}+|D|\mathbb{Z}_n$ is
$\frac{n}{|D|}.$ Therefore, by Proposition~\ref{pr-2.1},
$\mathbb{Z}_n=(\frac{n}{2}+|D|\mathbb{Z}_n)\oplus D.$ Set
$S=\frac{n}{2}+|D|\mathbb{Z}_n.$ Then $S$ is inverse-closed and
$\mathbb{Z}_n=S\oplus D.$ Next we claim that $0\notin S.$ Otherwise,
if $0\in S,$ then there exists some
$k\in\{0,1,\cdots,\frac{n}{|D|}-1\}$ such that $0=\frac{n}{2}+k|D|,$
and thus $k|D|=\frac{n}{2},$ which implies that $\frac{n}{|D|}=2k.$
Therefore, $|D|$ is not relatively prime to $\frac{n}{|D|},$ which
is a contradiction. By Lemma \ref{yl-2.2}, $D$ is an efficient open
dominating set in the circulant graph $\Cay(\mathbb{Z}_n,S).$

If $D$ is an efficient open dominating set in
$\Cay(\mathbb{Z}_n,S),$ then by the definition of an efficient open
dominating set, it is easy to see that $D$ induces a matching in
$\Cay(\mathbb{Z}_n,S),$ and thus $|D|$ is even. Moreover, if $D$ is
an efficient open dominating set in $\Cay(\mathbb{Z}_n,S),$ then by
Lemma \ref{yl-2.6}, $\mathbb{Z}_n=S\oplus D.$ Similar to the proof
of the `only if' part in Theorem \ref{dl-3.1}, we can obtain that
$d-d'\nequiv 0~(mod~|D|)$ for any distinct $d,d'\in D.$
\end{proof}

Using Lemmas \ref{yl-2.6} and \ref{yl-2.7} and following the proof
of Theorem \ref{dl-3.2} but with $S_0$ replaced by $S,$ one can
prove the following result, which also implies \cite[Theorem~1.4]{Feng} and \cite[Theorems~5.4--5.6]{Cam}.

\begin{theorem}\label{dl-3.7}
Let $S$ be an inverse-closed subset
of $\mathbb{Z}_n\setminus\{0\}$ such that $|S|$ divides $n$ and
$|S|$ is relatively prime to $\frac{n}{|S|}.$ Then
$\Cay(\mathbb{Z}_n,S)$ admits an efficient open dominating set if
and only if $s-s'\nequiv 0~(mod~|S|)$ for distinct $s,s'\in S.$
\end{theorem}

By Lemma \ref{yl-2.7} and Theorem \ref{dl-3.7}, one can easily
observe the following result.

\begin{corollary} \label{co-3.8}
Let $S$ be an inverse-closed
subset of $\mathbb{Z}_n\setminus\{0\}$ such that $|S|$ divides $n$
and $|S|$ is relatively prime to $\frac{n}{|S|}.$ Then
$\Cay(\mathbb{Z}_n,S)$ either doesn't admit any efficient open
dominating set or admits an efficient open dominating set which is a
subgroup of $\mathbb{Z}_n.$
\end{corollary}

Note that every finite abelian group is a direct product of some cyclic
groups. By Lemma~\ref{yl-2.6}, we can obtain the following
corollary which generalizes the result of Theorem~\ref{dl-3.7} about
circulant graphs by giving a similar result about Cayley graphs on
abelian groups.

\begin{corollary} \label{co-3.8}
Let \G be an abelian group and write
$G=\mathbb{Z}_{n_1}\times\cdots\times\mathbb{Z}_{n_k},$ where $k$
and $n_1,\cdots,n_k$ are positive integers, let
$S=S_1\times\cdots\times S_k,$ where $S_i$ is an inverse-closed
subset of $\mathbb{Z}_{n_i}\setminus\{0\}$ such that $|S_i|$ divides
$n_i$ and $|S_i|$ is relatively prime to $\frac{n_i}{|S_i|}$ for
each $i\in\{1,\cdots,k\}.$ Then $\Cay(G,S)$ admits an efficient open
dominating set if and only if $s_i-s_i'\nequiv 0~(mod~|S_i|)$ for distinct $s_i,s_i'\in S_i$ and each $i\in\{1,\cdots,k\}.$
\end{corollary}

Using Lemmas~\ref{yl-2.4} and \ref{yl-2.6} and following the proof
of Theorem~\ref{dl-3.4} but with $S_0$ replaced by $S$, we  can prove the following result.

\begin{theorem}\label{dl-3.11}
Let $S$ be an inverse-closed subset
of $\mathbb{Z}_n\setminus\{0\}$ such that $|S|$ divides $n$ and
$|S|$ is relatively prime to $\frac{n}{|S|}.$ Suppose that $\Cay(\mathbb{Z}_n,S)$
admits an efficient open dominating set. Then the set $\mathbb{D}$
of all efficient open dominating sets of $\Cay(\mathbb{Z}_n,S)$
satisfies the following (i)-(ii):

(i) $\{i+|S|\mathbb{Z}_n\colon\,i=0,1,\cdots,|S|-1\}\subseteq
\mathbb{D}\subseteq\{\{i+k_i\frac{n}{|S|}\colon\,i=0,1,\cdots,\frac{n}{|S|}-1\}
\colon\,\forall\,k_i{\in}\{0,1,\cdots,|S|{-}1\}\};$

(ii) For any $D\in\mathbb{D},\,l\in\mathbb{Z}_n,$ and any integer
$k$ relatively prime to $|D|,$ the set $kD+l\in\mathbb{D}.$
\end{theorem}

\section*{Conflict of interest}

The authors declare that there are no conflicts of interest.

\end{document}